\documentclass[a4paper,10pt,leqno]{amsart}
\usepackage{enumerate,amssymb}
\usepackage[arrow,curve,matrix,tips,2cell]{xy}
  \SelectTips{eu}{10} \UseTips
  \UseAllTwocells

\DeclareMathAlphabet{\matheurm}{U}{eur}{m}{n}

\DeclareMathOperator{\ANR}{ANR}

\DeclareMathOperator{\CAT}{CAT}
\DeclareMathOperator{\cd}{cd}

\DeclareMathOperator{\colim}{colim}
\DeclareMathOperator{\dom}{dom}

\DeclareMathOperator{\id}{id}

\DeclareMathOperator{\Inn}{Inn}

\DeclareMathOperator{\sign}{sign}

\DeclareMathOperator{\topo}{top}

\DeclareMathOperator{\Wh}{Wh}

  \newcommand{\IH}{\mathbb{H}}

  \newcommand{\IP}{\mathbb{P}}
  \newcommand{\IQ}{\mathbb{Q}}
  \newcommand{\IR}{\mathbb{R}}

  \newcommand{\IZ}{\mathbb{Z}}

  \newcommand{\calc}{\mathcal{C}}

  \newcommand{\call}{\mathcal{L}}
  
  \newcommand{\caln}{\mathcal{N}}

  \newcommand{\cals}{\mathcal{S}}

  \newcommand{\bff}{{\mathbf f}}

  \newcommand{\bfL}{{\mathbf L}}

\newcounter{commentcounter}

\theoremstyle{plain}
\newtheorem{theorem}{Theorem}[section]

\newtheorem{lemma}[theorem]{Lemma}

\newtheorem{proposition}[theorem]{Proposition}
\newtheorem{conjecture}[theorem]{Conjecture}

\theoremstyle{definition}
\newtheorem{definition}[theorem]{Definition}

\newtheorem{question}[theorem]{Question}

\newtheorem{remark}[theorem]{Remark}

\theoremstyle{remark}

\makeatletter\let\c@equation=\c@theorem\makeatother

\title{Survey on aspherical manifolds}
              \author{Wolfgang L\"uck}
      \address{Westf\"alische Wilhelms-Universit\"at M\"unster\\
               Mathematisches Institut\\
               Einsteinstr.~62,
               D-48149 M\"unster, Germany}
                \email{lueck@math.uni-muenster.de}
      \urladdr{http://www.math.uni-muenster.de/u/lueck}
         \date{July 2009}

\begin{document}

\maketitle

\begin{abstract}
  This is a survey on known results and open problems about closed
  aspherical manifolds, i.e., connected closed manifolds whose
  universal coverings are contractible.  Many examples come from
  certain kinds of non-positive curvature conditions.  The
  property aspherical, which is a purely homotopy theoretical
  condition, implies many striking results about the geometry and
  analysis of the manifold or its universal covering, and the ring
  theoretic properties and the $K$- and $L$-theory of the group ring
  associated to its fundamental group.  The Borel Conjecture predicts
  that closed aspherical manifolds are topologically rigid. The article contains
  new results about product decompositions of closed aspherical manifolds and 
  an announcement of a result joint with Arthur Bartels  and Shmuel Weinberger
  about hyperbolic groups with spheres of dimension $\ge 6$ as boundary.
  At the end we describe (winking) our universe of closed manifolds.

\end{abstract}


\typeout{--------------------   Section 0: Introduction --------------------------}

\setcounter{section}{-1}
\section{Introduction}
\label{sec:Introduction}

A space $X$ is called \emph{aspherical} if it is path connected and
all its higher homotopy groups vanish, i.e., $\pi_n(X)$ is trivial for
$n \ge 2$.  This survey article is devoted to aspherical closed
manifolds. These are very interesting objects for many reasons. Often
interesting geometric constructions or examples lead to aspherical
closed manifolds. The study of the question which groups occur as
fundamental groups of closed aspherical manifolds is
intriguing. The condition aspherical is of purely homotopy
theoretical nature. Nevertheless there are some interesting questions and conjectures
about curvature properties of a closed aspherical Riemann manifold
and about the spectrum of the Laplace operator on its
universal covering. The Borel Conjecture predicts that aspherical
closed topological manifolds are topologically rigid and that
aspherical compact Poincar\'e complexes are homotopy equivalent to
closed manifolds. We discuss the status of some of these questions and
conjectures. Examples of exotic aspherical closed manifolds come from
hyperbolization techniques and we list certain examples.
At the end we describe (winking) our universe of closed manifolds.

The results about product decompositions of closed aspherical manifolds
in Section~\ref{sec:Product_decompositions}
are new and Section~\ref{sec:Boundaries_of_hyperbolic_groups} contains
an announcement of a result joint with Arthur Bartels and Shmuel Weinberger 
about hyperbolic groups with spheres of dimension $\ge 6$ as boundary.

The author wants to the thank the Max-Planck-Institute for Mathematics
in Bonn for its hospitality during his stay from October 2007 until
December 2007 when parts of this paper were written.  The work was
financially supported by the Sonderforschungsbereich 478 \---
Geometrische Strukturen in der Mathematik \--- the
Max-Planck-Forschungspreis and the Leibniz-Preis of the author.
The author wants to thank the referee for his valuable suggestions.

The paper is organized as follows:

\tableofcontents


\typeout{--------------------   Section 1 ---------------------------------------}

\section{Homotopy theory of aspherical manifolds}
\label{sec:Homotopy_theory_of_aspherical_manifolds}

From the homotopy theory point of view an aspherical $CW$-complex
is completely determined by its fundamental group. Namely

\begin{theorem}[Homotopy classification of aspherical spaces] \
  \label{the:homotopy_class_asph_CW}
  \begin{enumerate}

  \item \label{the:homotopy_class_asph_CW:homotopy_equivalence} Two
    aspherical $CW$-complexes are homotopy equivalent if and only if
    their fundamental groups are isomorphic;

  \item \label{the:homotopy_class_asph_CW:homotopy_classes} Let $X$
    and $Y$ be connected $CW$-complexes. Suppose that $Y$ is
    aspherical.  Then we obtain a bijection
    $$[X,Y] \xrightarrow{\cong} [\Pi(X),\Pi(Y)],
    \quad [f] \mapsto [\Pi(f))],$$ 
    where $[X,Y]$ is the set of homotopy
    classes of maps from $X$ to $Y$, $\Pi(X)$, $\Pi(Y)$ are the
    fundamental groupoids, $[\Pi(X),\Pi(Y)]$ is the set of natural 
    equivalence classes of functors from $\Pi(X)$ to $\Pi(Y)$
    and $\Pi(f) \colon \Pi(X) \to \Pi(Y)$ is the functor
    induced by $f \colon X \to Y$.
\end{enumerate}
\end{theorem}
\begin{proof}~\ref{the:homotopy_class_asph_CW:homotopy_classes} 
  One easily checks
  that the map is well-defined.  For the proof of surjectivity and
  injectivity one constructs the desired preimage or the desired
  homotopy inductively over the skeletons of the source.
  \\[1mm]\ref{the:homotopy_class_asph_CW:homotopy_equivalence} This
  follows directly from
  assertion~\ref{the:homotopy_class_asph_CW:homotopy_classes}.
\end{proof}

The description using fundamental groupoids is elegant and base point
free, but a reader may prefer its more concrete interpretation in
terms of fundamental groups, which we will give next:  Choose base
points $x \in X$ and $y \in Y$. Let $\hom(\pi_1(X,x),\pi_1(Y,y))$ be
the set of group homomorphisms from $\pi_1(X,x)$ to $\pi_1(Y,y)$.  The
group $\Inn\bigl(\pi_1(Y,y)\bigr)$ of inner automorphisms of
$\pi_1(Y,y)$ acts on $\hom\bigl(\pi_1(X,x),\pi_1(Y,y)\bigr)$ from the
left by composition.  We leave it to the reader to check that we
obtain a bijection
$$\Inn\bigl(\pi_1(Y,y)\bigr)\backslash
\hom\bigl(\pi_1(X,x),\pi_1(Y,y)\bigr) \xrightarrow{\cong}
[\Pi(X),\Pi(Y)],$$ under which the bijection appearing in
Lemma~\ref{the:homotopy_class_asph_CW}~%
\ref{the:homotopy_class_asph_CW:homotopy_classes} sends $[f]$ to the
class of $\pi_1(f,x)$ for any choice of representative of $f$ with
$f(x) = y$. In the sequel we will often ignore base points especially
when dealing with the fundamental group.

\begin{lemma} \label{lem:characterization_of_aspherical_complexes} A
  $CW$-complex $X$ is aspherical if and only if it is connected and
  its universal covering $\widetilde{X}$ is contractible.
\end{lemma}
\begin{proof} The projection $p \colon \widetilde{X} \to X$ induces
  isomorphisms on the homotopy groups $\pi_n$ for $n \ge 2$ and a
  connected $CW$-complex is contractible if and only if all its
  homotopy groups are trivial 
  (see\cite[Theorem~IV.7.17 on page~182]{Whitehead(1978)}.
\end{proof}

An aspherical $CW$-complex $X$ with fundamental group $\pi$ is the
same as an \emph{Eilenberg Mac-Lane space $K(\pi,1)$ of type
$(\pi,1)$} and the same as the \emph{classifying space $B\pi$ for
the group $\pi$}.


\typeout{--------------------   Section 2 ---------------------------------------}

\section{Examples of aspherical manifolds}
\label{sec:Examples_of_aspherical_manifolds}

In this section we give examples and constructions  of aspherical closed manifolds.

\subsection{Non-positive curvature}
\label{subsec:Non-positive_curvature}

Let $M$ be a closed smooth manifold. Suppose that it possesses a Riemannian
metric whose sectional curvature is non-positive, i.e., is~$\le 0$ everywhere.
Then the  universal covering $\widetilde{M}$ inherits a complete Riemannian metric
whose sectional curvature is non-positive.
Since $\widetilde{M}$ is
simply-connected and has non-positive sectional curvature, 
the Hadamard-Cartan Theorem 
(see~\cite[3.87 on page~134]{Gallot-Hulin-Lafontaine(1987)})
implies that $\widetilde{M}$
is diffeomorphic to $\IR^n$ and hence contractible. We conclude
that $\widetilde{M}$ and hence $M$ is aspherical.

\subsection{Low-dimensions}
\label{subsec:Low-dimensions}

A connected closed $1$-dimensional manifold is homeomorphic to $S^1$
and hence aspherical.

Let $M$ be a connected closed $2$-dimensional manifold. 
Then $M$  is either aspherical or
homeomorphic to $S^2$ or $\IR\IP^2$. The following statements are
equivalent:
i.) $M$ is aspherical.\ ii.) $M$  admits a Riemannian metric which is
\emph{flat}, i.e., with sectional curvature constant  $0$, or which
is \emph{hyperbolic}, i.e., with sectional curvature constant 
$-1$.\  iii) The universal covering of $M$ is homeomorphic to $\IR^2$. 

A connected closed $3$-manifold $M$ is called \emph{prime} if for any
decomposition as a connected sum $M \cong M_0 \sharp M_1$ one of the
summands $M_0$ or $M_1$ is homeomorphic to~$S^3$. It is called
\emph{irreducible} if any embedded sphere $S^2$ bounds a disk
$D^3$. Every irreducible closed $3$-manifold is prime. A prime closed
$3$-manifold is either irreducible or an $S^2$-bundle over $S^1$
(see~\cite[Lemma~3.13 on page~28]{Hempel(1976)}). A
closed orientable $3$-manifold is aspherical if and only if it is
irreducible and has infinite fundamental group. A closed $3$-manifold
is aspherical if and only if it is irreducible and its fundamental
group is infinite and contains no element of order $2$.
This follows from the Sphere Theorem~\cite[Theorem~4.3 on
page~40]{Hempel(1976)}.

Thurston's Geometrization Conjecture implies that a closed
$3$-manifold is aspherical if and only if its universal covering is
homeomorphic to $\IR^3$. This follows from~\cite[Theorem~13.4 on
page~142]{Hempel(1976)} and the fact that the
$3$-dimensional geometries which have compact quotients and 
whose underlying topological spaces are contractible
have as underlying
smooth manifold $\IR^3$ (see~\cite{Scott(1983)}).

A proof of Thurston's Geometrization Conjecture
is given in~\cite{Morgan-Tian(2008)} following ideas of Perelman.

There are examples of closed orientable $3$-manifolds that are
aspherical but do not support a Riemannian metric with non-positive
sectional curvature (see~\cite{Leeb(1995)}).

For more information about $3$-manifolds we refer for instance
to~\cite{Hempel(1976),Scott(1983)}.

\subsection{Torsionfree discrete subgroups of almost connected Lie
  groups}
\label{subsec:Torsionfree_discrete_subgroups_of_almost_connected_Lie_groups}

Let $L$ be a Lie group with finitely many path components. Let $K
\subseteq L$ be a maximal compact subgroup. Let $G \subseteq L$ be a
discrete torsionfree subgroup.  Then $M = G\backslash L/K$ is a closed
aspherical manifold with fundamental group $G$ since its universal
covering $L/K$ is diffeomorphic to $\IR^n$ for appropriate $n$
(see~\cite[Theorem 1. in Chapter VI]{Helgason(2001)}).

\subsection{Hyperbolization}
\label{subsec:hyperbolization}
A very important construction of aspherical manifolds comes from the
\emph{hyperbolization technique} due to Gromov~\cite{Gromov(1987)}. It
turns a cell complex into a non-positively curved (and hence
aspherical) polyhedron. The rough idea is to define this procedure for
simplices such that it is natural under inclusions of simplices and
then define the hyperbolization of a simplicial complex by gluing the
results for the simplices together as described by the combinatorics
of the simplicial complex.  The goal is to achieve that the result
shares some of the properties of the simplicial complexes one has
started with, but additionally to produce a non-positively curved and
hence aspherical polyhedron.  Since this construction preserves local
structures, it turns manifolds into manifolds.

We briefly explain what the \emph{orientable hyperbolization
  procedure} gives.  Further expositions of this construction can be
found in~\cite{Charney-Davis(1995),Davis(2002exotic),Davis(2008cox),Davis-Januszkiewicz(1991)}.
We start with a finite-dimensional simplicial complex $\Sigma$ and a assign to
it a cubical cell complex $h(\Sigma)$ and a natural map $c \colon h(\Sigma) \to \Sigma$
with the following properties:
\begin{enumerate}

\item $h(\Sigma)$ is non-positively curved and in particular
  aspherical;

\item The natural map $c \colon h(\Sigma) \to \Sigma$ induces a
  surjection on the integral homology;

\item $\pi_1(f) \colon \pi_1(h(\Sigma)) \to \pi_1(\Sigma)$ is
  surjective;

\item If $\Sigma$ is an orientable manifold, then

  \begin{enumerate}

  \item $h(\Sigma)$ is a manifold;

  \item The natural map $c \colon h(\Sigma) \to \Sigma$ has degree
    one;

  \item There is a stable isomorphism between the tangent bundle
    $Th(\Sigma)$ and the pullback $c^*T\Sigma$;

  \end{enumerate}
\end{enumerate}

\begin{remark}[Characteristic numbers and aspherical manifolds]
  \label{rem:Characteristic_numbers_and_aspherical_manifolds}
  Suppose that $M$ is a closed manifold. Then the 
  pullback of the characteristic classes of $M$ under the natural map
  $c \colon h(M) \to M$ yield the characteristic classes of $h(M)$,
  and $M$ and $h(M)$ have the same characteristic numbers. This shows
  that the condition aspherical does not impose any restrictions on
  the characteristic numbers of a manifold.
\end{remark}

\begin{remark}[Bordism and aspherical manifolds]
  \label{rem:bordism_and_aspherical_manifolds}
  The conditions above say that $c$ is a normal map in the sense of
  surgery. One can show that $c$ is normally bordant to the identity
  map on $M$. In particular $M$ and $h(M)$ are oriented bordant.

  Consider a bordism theory $\Omega_*$ for PL-manifolds or smooth
  manifolds which is given by imposing conditions on the stable
  tangent bundle. Examples are unoriented bordism, oriented bordism,
  framed bordism. Then any bordism class can be represented by an
  aspherical closed manifold. If two closed aspherical manifolds
  represent the same bordism class, then one can find an aspherical
  bordism between them.  See~\cite[Remarks~15.1]{Davis(2002exotic)}
  and~\cite[Theorem~B]{Davis-Januszkiewicz(1991)}.
\end{remark}

\subsection{Exotic aspherical manifolds}
\label{subsec:exotic_asphercial_manifolds}

The following result is taken from
Davis-Januszkiewicz~\cite[Theorem~5a.1]{Davis-Januszkiewicz(1991)}.
\begin{theorem}\label{the:exotic_aspherical_manifolds_in_dimension_4}
  There is a closed aspherical $4$-manifold $N$ with the following
  properties:

  \begin{enumerate}

  \item $N$ is not homotopy equivalent to a $PL$-manifold;

  \item $N$ is not triangulable, i.e., not homeomorphic to a
    simplicial complex;

  \item The universal covering $\widetilde{N}$ is not homeomorphic to
    $\IR^4$;

  \item $N$ is homotopy equivalent to a piecewise flat, non-positively
    curved polyhedron.

  \end{enumerate}
\end{theorem}

The next result is due to 
Davis-Januszkiewicz~\cite[Theorem~5a.4]{Davis-Januszkiewicz(1991)}.

\begin{theorem}[Non-PL-example] \label{the:universal_covering_not_PL}
  For every $n \ge 4$ there exists a closed aspherical $n$-manifold
  which is not homotopy equivalent to a PL-manifold
\end{theorem}

The proof of the following theorem can be found 
in~\cite{Davis(1983)},~\cite[Theorem~5b.1]{Davis-Januszkiewicz(1991)}.

\begin{theorem}[Exotic universal
  covering] \label{the:universal_covering_notIR} For each $n \ge 4$
  there exists a closed aspherical $n$-dimensional 
  manifold such that its universal
  covering is not homeomorphic to $\IR^n$.
\end{theorem}

By the Hadamard-Cartan Theorem (see~\cite[3.87 on
page~134]{Gallot-Hulin-Lafontaine(1987)}) the manifold appearing in
Theorem~\ref{the:universal_covering_notIR} above cannot be
homeomorphic to a smooth manifold with Riemannian metric with
non-positive sectional curvature.

The following theorem is proved in~\cite[Theorem~5c.1 and Remark on page
  386]{Davis-Januszkiewicz(1991)} by considering the ideal boundary,
  which is a quasiisometry invariant in the negatively curved case.

\begin{theorem}[Exotic example with hyperbolic fundamental group]
  \label{the:asheprcial_hyperbolic_fundmanetal_group}
  For every $n \ge 5$ there exists an aspherical closed smooth
  $n$-dimensional manifold $N$ which is homeomorphic to a strictly
  negatively curved polyhedron and has in particular a hyperbolic
  fundamental group such that the universal covering is homeomorphic
  to $\IR^n$ but $N$ is not homeomorphic to a smooth manifold with
  Riemannian metric with negative sectional curvature.
\end{theorem}

The next results are due to 
Belegradek~\cite[Corollary~5.1]{Belegradek(2006)},
Mess~\cite{Mess(1991)} and Weinberger
(see~\cite[Section~13]{Davis(2002exotic)}).

\begin{theorem}[Exotic fundamental groups]\
  \label{the:Exotic_fundamental_groups}

  \begin{enumerate}
    \item For every $n \ge 4$ there is a closed aspherical manifold of
    dimension $n$ whose fundamental group contains an infinite
    divisible abelian group;

  \item For every $n \ge 4$ there is a closed aspherical manifold of
    dimension $n$ whose fundamental group has an unsolvable word
    problem and whose simplicial volume is non-zero.
  \end{enumerate}
\end{theorem}

Notice that a finitely presented group with unsolvable word problem is
not a $\CAT(0)$-group, not hyperbolic,
not automatic, not asynchronously automatic, 
not residually finite and not linear over any commutative ring
(see~\cite[Remark~5.2]{Belegradek(2006)}).

The proof of Theorem~\ref{the:Exotic_fundamental_groups} is based
on the \emph{reflection group trick} as it appears for instance 
in~\cite[Sections~8,10 and~13]{Davis(2002exotic)}. It can be summarized as follows.

\begin{theorem}[Reflection group trick]
\label{the:Reflection_group_trick}
Let $G$ be a group which  possesses a finite model for $BG$.
Then there is a closed aspherical manifold $M$ and a map
$i \colon BG \to M$ and $r \colon M \to BG$ such that $r \circ i = \id_{BG}$.
\end{theorem}

\begin{remark}[Reflection group trick and various conjectures]
\label{rem:Reflection_group_trick_and_various_conjectures}
Another interesting immediate consequence
of the reflection group trick is
(see also~\cite[Sections~11]{Davis(2002exotic)})
that many well-known conjectures about groups 
hold for every group which possesses a finite model for $BG$ if and only if it holds for
the fundamental group of every closed aspherical manifold. This applies for instance
to the Kaplansky Conjecture, Unit Conjecture, Zero-divisor-conjecture, 
Baum-Connes Conjecture, Farrell-Jones Conjecture
for algebraic $K$-theory for regular $R$, Farrell-Jones Conjecture
for algebraic $L$-theory, the vanishing of $\widetilde{K}_0(\IZ G)$ and of $\Wh(G) = 0$,
For information about these conjectures and their links we refer
for instance to~\cite{Bartels-Lueck-Reich(2008appl)},\cite{Lueck(2002)}
and~\cite{Lueck-Reich(2005)}. Further similar consequences of the reflection group trick
can be found in Belegradek~\cite{Belegradek(2006)}.

\end{remark}


\typeout{--------------------   Section 3 ---------------------------------------}

\section{Non-aspherical closed manifolds}
\label{sec:Non-aspherical_closed_manifolds}

A closed manifold of dimension $\ge 1$ with finite fundamental group is never aspherical.
So prominent non-aspherical manifolds are spheres, lens spaces, real projective spaces
and complex projective spaces. 

\begin{lemma}
  \label{lem:torsion_in_pi}
  The fundamental group of an aspherical finite-dimensional
  $CW$-complex $X$ is torsionfree.
\end{lemma}
\begin{proof}
  Let $C \subseteq \pi_1(X)$ be a finite cyclic subgroup of
  $\pi_1(X)$.  We have to show that $C$ is trivial. Since $X$ is
  aspherical, $C\backslash \widetilde{X}$ is a finite-dimensional
  model for $BC$. Hence $H_k(BC)=  0$ for large $k$. This implies that
  $C$ is trivial.
\end{proof}

\begin{lemma}
  \label{lem:connected_sum_not_aspherical}
  If $M$ is a connected sum $M_1 \sharp M_2$ of two closed manifolds
  $M_1$ and $M_2$ of dimension $n \ge 3$ which are not homotopy
  equivalent to a sphere, then $M$ is not aspherical.
\end{lemma}
\begin{proof}
  We proceed by contradiction. Suppose that $M$ is aspherical.  The
  obvious map $f \colon M_1 \sharp M_2 \to M_1 \vee M_2$ given by
  collapsing $S^{n-1}$ to a point is $(n-1)$-connected, where $n$ is
  the dimension of $M_1$ and $M_2$.  Let $p \colon \widetilde{M_1 \vee
    M_2} \to M_1 \vee M_2$ be the universal covering. By the
  Seifert-van Kampen Theorem the fundamental group of $\pi_1(M_1\vee
  M_2)$ is $\pi_1(M_1) \ast \pi_1(M_2)$ and the inclusion of $M_k \to
  M_1 \vee M_2$ induces injections on the fundamental groups for $k =
  1,2$.  We conclude that $p^{-1}(M_k) = \pi_1(M_1 \vee M_2) \times_{\pi_1(M_k)}
  \widetilde{M_k}$ for $k = 1,2$.  Since $n \ge 3$, the map $f$
  induces an isomorphism on the fundamental groups and an
  $(n-1)$-connected map $\widetilde{f} \colon \widetilde{M_1 \sharp
    M_2} \to \widetilde{M_1 \vee M_2}$.  Since $\widetilde{M_1 \sharp
    M_2}$ is contractible, $H_m(\widetilde{M_1 \vee M_2}) = 0$ for $1
  \le m \le n-1$.  Since $p^{-1}(M_1) \cup p^{-1}(M_2) =
  \widetilde{M_1 \vee M_2}$ and $p^{-1}(M_1) \cap p^{-1}(M_2) =
  p^{-1}(\{\bullet\}) = \pi_1(M_1 \vee M_2)$, we conclude
  $H_m(p^{-1}(M_k)) = 0$ for $1 \le m \le n-1$ from the Mayer-Vietoris
  sequence.  This implies $H_m(\widetilde{M_k}) = 0$ for $1 \le m \le
  n-1$ since $p^{-1}(M_k)$ is a disjoint union of copies of $\widetilde{M_k}$. 

  Suppose that $\pi_1(M_k)$ is finite. Since $\pi_1(M_1 \sharp M_2)$
  is torsionfree by Lemma~\ref{lem:torsion_in_pi}, $\pi_1(M_k)$ must
  be trivial and $M_k = \widetilde{M_k}$.  Since $M_k$ is simply
  connected and $H_m(M_k) = 0$ for $1 \le m \le n-1$, $M_k$ is
  homotopy equivalent to $S^n$. Since we assume that $M_k$ is not
  homotopy equivalent to a sphere, $\pi_1(M_k)$ is infinite.  This
  implies that the manifold $\widetilde{M_k}$ is non-compact and hence
  $H_n(\widetilde{M_k}) = 0$.  Since $\widetilde{M_k}$ is
  $n$-dimensional, we conclude $H_m(\widetilde{M_k}) = 0$ for $m \ge 1$.
  Since $\widetilde{M_k}$ is simply connected, all homotopy groups
  of $\widetilde{M_k}$ vanish by the Hurewicz
  Theorem~\cite[Corollary~IV.7.8 on page~180]{Whitehead(1978)}. We
  conclude from
  Lemma~\ref{lem:characterization_of_aspherical_complexes} that $M_1$
  and $M_2$ are aspherical. Using the Mayer-Vietoris argument above
  one shows analogously that $M_1 \vee M_2$ is aspherical. Since $M$
  is by assumption aspherical, $M_1\sharp M_2$ and $M_1 \vee M_2$ are
  homotopy equivalent by Lemma~\ref{the:homotopy_class_asph_CW}~%
\ref{the:homotopy_class_asph_CW:homotopy_equivalence}. Since they
  have different Euler characteristics, namely $\chi(M_1 \sharp M_2) =
  \chi(M_1) + \chi(M_2) - (1 + (-1)^n)$ and $\chi(M_1 \vee M_2) =
  \chi(M_1) + \chi(M_2) - 1$, we get a contradiction.
\end{proof}


\typeout{--------------------   Section 4 ---------------------------------------}

\section{The Borel Conjecture}
\label{sec:Borel_Conjecture}

In this section we deal with

\begin{conjecture}[Borel Conjecture for a group $G$]
  \label{con:Borel_Conjecture}
  If $M$ and $N$ are closed aspherical manifolds of dimensions $\ge 5$
  with $\pi_1(M) \cong \pi_1(N) \cong G$, then $M$ and $N$ are
  homeomorphic and any homotopy equivalence $M \to N$ is homotopic to
  a homeomorphism.
\end{conjecture}

\begin{definition}[Topologically rigid]
  \label{def:topological_rigid}
   We call a closed manifold $N$ 
  \emph{topologically rigid} if any homotopy equivalence $M \to N$
   with a closed manifold $M$ as source is homotopic to a homeomorphism.
 \end{definition}
 
If the Borel Conjecture holds for all finitely presented groups,
then every closed aspherical manifold is topologically rigid.

The main tool to attack the Borel Conjecture is surgery theory and the
Farrell-Jones Conjecture. We consider the following special version of
the Farrell-Jones Conjecture.

\begin{conjecture}[Farrell-Jones Conjecture for torsionfree groups and
  regular rings]
  \label{con:Farrell-Jones_Conjecture_torsionfree}
  Let $G$ be a torsionfree group and let $R$ be a regular ring, e.g.,
  a principal ideal domain, a field, or $\IZ$. Then
  \begin{enumerate}
  \item $K_{n}(RG) = 0$ for $n \le -1$;
  \item The change of rings homomorphism $K_0(R) \to K_0(RG)$ is
    bijective.  (This implies in the case $R = \IZ$ that the reduced
    projective class group $\widetilde{K}_0(\IZ G)$ vanishes;
  \item The obvious map $K_1(R) \times G/[G,G] \to K_1(RG)$ is
    surjective.  (This implies in the case $R = \IZ$ that the
    Whitehead group $\Wh(G)$ vanishes);
  \item For any orientation homomorphism $w \colon G \to \{\pm 1\}$
    the $w$-twisted $L$-theoretic assembly map
$$H_n(BG;^w\bfL^{\langle -\infty\rangle}) \xrightarrow{\cong} L^{\langle -\infty \rangle}_n(RG,w)$$
is bijective.
\end{enumerate}
\end{conjecture}

\begin{lemma}
  \label{lem:FJC_implies_BC}
  Suppose that the torsionfree group $G$ satisfies the version of the
  Farrell-Jones Conjecture stated in
  Conjecture~\ref{con:Farrell-Jones_Conjecture_torsionfree} for $R =
  \IZ$.

  Then the Borel Conjecture is true for closed aspherical
  manifolds of dimension $\ge 5$ with $G$ as fundamental group.  Its is true for closed aspherical
  manifolds of dimension $4$  with $G$ as fundamental group
  if $G$ is good in the sense of
  Freedman (see~\cite{Freedman(1982)}, \cite{Freedman(1983)}).
\end{lemma}
\begin{proof}[Sketch of the proof]
  We treat the orientable case only.  The \emph{topological structure
    set} $\cals^{\topo}(M)$ of a closed topological manifold $M$ is
  the set of equivalence classes of homotopy equivalences $M' \to M$
  with a topological closed manifold as source and $M$ as target under
  the equivalence relation, for which $f_0 \colon M_0 \to M$ and $f_1
  \colon M_1 \to M$ are equivalent if there is a homeomorphism
  $g\colon M_0 \to M_1$ such that $f_1 \circ g$ and $f_0$ are
  homotopic.  The Borel Conjecture~\ref{con:Borel_Conjecture} for a
  group $G$ is equivalent to the statement that for every closed
  aspherical manifold $M$ with $G \cong \pi_1(M)$ its topological
  structure set $\cals^{\topo}(M)$ consists of a single element,
  namely, the class of $\id \colon M \to M$.

  The \emph{surgery sequence} of a closed orientable topological
  manifold $M$ of dimension $n \ge 5$ is the exact sequence
  \begin{multline*}
    \ldots \to \caln_{n+1}\bigl(M\times [0,1],M \times \{0,1\}\bigr)
    \xrightarrow{\sigma} L^s_{n+1}\bigl(\IZ\pi_1(M)\bigr)
    \xrightarrow{\partial} \cals^{\topo}(M)
    \\
    \xrightarrow{\eta} \caln_n(M) \xrightarrow{\sigma}
    L_n^s\bigl(\IZ\pi_1(M)\bigr),
  \end{multline*}
  which extends infinitely to the left.  It is the basic tool for the
  classification of topological manifolds.  (There is also a smooth
  version of it.)  The map $\sigma$ appearing in the sequence sends a
  normal map of degree one to its surgery obstruction.  This map can
  be identified with the version of the $L$-theory assembly map where
  one works with the $1$-connected cover $\bfL^s ( \IZ ) \langle 1
  \rangle$ of $\bfL^s( \IZ )$.  The map $H_k\bigl(M;\bfL^s
  (\IZ)\langle 1 \rangle \bigr) \to H_k\bigl(M;\bfL^s (\IZ)\bigr)$ is
  injective for $k=n$ and an isomorphism for $k >n$. Because of the
  $K$-theoretic assumptions we can replace the $s$-decoration with the
  $\langle - \infty \rangle$-decoration. Therefore the Farrell-Jones
  Conjecture implies that the maps $\sigma\colon \caln_n(M) \to
  L_n^s\bigl(\IZ\pi_1(M)\bigr)$ and $\caln_{n+1}\bigl(M\times [0,1],M
  \times \{0,1\}\bigr) \xrightarrow{\sigma}
  L^s_{n+1}\bigl(\IZ\pi_1(M)\bigr)$ are injective respectively
  bijective and thus by the surgery sequence that $\cals^{\topo}(M)$
  is a point and hence the Borel Conjecture~\ref{con:Borel_Conjecture}
  holds for $M$.  More details can be found e.g., in \cite[pages
  17,18,28]{Ferry-Ranicki-Rosenberg(1995)}, \cite[Chapter
  18]{Ranicki(1992)}.
\end{proof}

\begin{remark}[The Borel Conjecture in low dimensions]
  \label{rem:The_Borel_Conjecture_in_low_dimensions}
  The Borel Conjecture is true in dimension $\le 2$ by the
  classification of closed manifolds of dimension $2$. It is true in
  dimension $3$ if Thurston's Geometrization Conjecture is true.  This
  follows from results of Waldhausen (see Hempel~\cite[Lemma~10.1 and
  Corollary~13.7]{Hempel(1976)}) and Turaev (see~\cite{Turaev(1988)})
  as explained for instance
  in~\cite[Section~5]{Kreck-Lueck(2009nonasph)}.  A proof of Thurston's
  Geometrization Conjecture is given in~\cite{Morgan-Tian(2008)}
  following ideas of Perelman.
\end{remark}

\begin{remark}[Topological rigidity for non-aspherical manifolds]
  \label{rem:Topological_rigidity_for_non-aspherical_manifolds}
  Topological rigidity phenomenons do hold also for some
  non-aspherical closed manifolds.
  For instance the sphere $S^n$ is
  topologically rigid by the Poincar\'e Conjecture. 
  The Poincar\'e Conjecture is known to be true in all dimensions.
  This follows in high dimensions from the $h$-cobordism theorem,
  in dimension four from the work of Freedman~\cite{Freedman(1982)},
  in dimension three from the work of Perelman as explained 
  in~\cite{Kleiner-Lott(2006),Morgan+Tian(2008_Ricci_poincare)} and
  and in dimension two from the classification of surfaces.

  Many more examples of classes of manifolds which are topologically
  rigid  
  are given and analyzed in
  Kreck-L\"uck~\cite{Kreck-Lueck(2009nonasph)}.  For instance the connected
  sum of closed manifolds of dimension $\ge 5$ which are
  topologically rigid and whose fundamental groups
  do not contain elements of order two, is again topologically rigid
  and the connected sum of two manifolds is in general not aspherical
  (see~Lemma~\ref{lem:connected_sum_not_aspherical}). The product $S^k
  \times S^n$ is topologically rigid if and only if $k$ and $n$ are
  odd. An integral homology sphere of dimension $n \ge 5$ is
  topologically rigid if and only if the inclusion $\IZ \to
  \IZ[\pi_1(M)]$ induces an isomorphism of simple $L$-groups
  $L_{n+1}^s(\IZ) \to L_{n+1}^s\bigl(\IZ[\pi_1(M)]\bigr)$.
\end{remark}

\begin{remark}[The Borel Conjecture does not hold in the smooth
  category]
  \label{rem:The_Borel_Conjecture_does_not_hold_in_the_smooth_category}
  The Borel Conjecture~\ref{con:Borel_Conjecture} is false in the
  smooth category, i.e., if one replaces topological manifold by
  smooth manifold and homeomorphism by diffeomorphism. The torus $T^n$
  for $ n \ge 5$ is an example (see~\cite[15A]{Wall(1999)}).  Other
  counterexample involving negatively curved manifolds are constructed
  by Farrell-Jones~\cite[Theorem~0.1]{Farrell-Jones(1989b)}.
\end{remark}

\begin{remark}[The Borel Conjecture versus Mostow rigidity]
  \label{rem:The_Borel_Conjecture_versus_Mostow_rigidity}
  The examples of
  Farrell-Jones~\cite[Theorem~0.1]{Farrell-Jones(1989b)} give actually
  more.  Namely, it yields for given $\epsilon > 0$ a closed
  Riemannian manifold $M_0$ whose sectional curvature lies in the
  interval $[1-\epsilon,-1 + \epsilon]$ and a closed hyperbolic
  manifold $M_1$ such that $M_0$ and $M_1$ are homeomorphic but no
  diffeomorphic.  The idea of the construction is essentially to take
  the connected sum of $M_1$ with exotic spheres.  Notice that by
  definition $M_0$ were hyperbolic if we would take $\epsilon = 0$.
  Hence this example is remarkable in view of \emph{Mostow rigidity},
  which predicts for two closed hyperbolic manifolds $N_0$ and $N_1$
  that they are isometrically diffeomorphic if and only if $\pi_1(N_0)
  \cong \pi_1(N_1)$ and any homotopy equivalence $N_0 \to N_1$ is
  homotopic to an isometric diffeomorphism.

  One may view the Borel Conjecture as the topological version of
  Mostow rigidity.  The conclusion in the Borel Conjecture is weaker,
  one gets only homeomorphisms and not isometric diffeomorphisms, but
  the assumption is also weaker, since there are many more aspherical
  closed topological manifolds than hyperbolic closed manifolds.
\end{remark}

\begin{remark}[The work of Farrell-Jones]\label{rem:work_of_Farrell-Jones} 
  Farrell-Jones have made deep contributions to the Borel Conjecture.
  They have proved it in dimension $\ge 5$ for non-positively curved
  closed Riemannian manifolds, for compact complete affine flat
  manifolds and for closed aspherical manifolds whose fundamental
  group is isomorphic to the fundamental group of a complete
  non-positively curved Riemannian manifold which is A-regular
  (see~\cite{Farrell-Jones(1990),Farrell-Jones(1990b),Farrell-Jones(1993c),Farrell-Jones(1998)}).
\end{remark}

The following result is due to Bartels and
L\"uck~\cite{Bartels-Lueck(2009Borelhyp)}.
\begin{theorem}
  \label{the:status_of-Farrell-Jones}
  Let $\calc$ be the smallest class of groups satisfying:
  \begin{itemize}

  \item Every hyperbolic group belongs to $\calc$;

  \item Every group that acts properly, isometrically and cocompactly on a
    complete proper $\CAT(0)$-space belongs to $\calc$;

  \item If $G_1$ and $G_2$ belong to $\calc$, then both $G_1 \ast G_2$
    and $G_1 \times G_2$ belong to $\calc$;

  \item If $H$ is a subgroup of $G$ and $G \in \calc$, then $H \in
    \calc$;

  \item Let $\{G_i \mid i \in I\}$ be a directed system of groups
    (with not necessarily injective structure maps) such that $G_i \in
    \calc$ for every $i \in I$. Then the directed colimit $\colim_{i \in I}
    G_i$ belongs to $\calc$.

  \end{itemize}

  Then every group $G$ in $\calc$ satisfies the version of the
  Farrell-Jones Conjecture stated in
  Conjecture~\ref{con:Farrell-Jones_Conjecture_torsionfree}.
\end{theorem}

\begin{remark}[Exotic closed aspherical manifolds]
\label{rem:exotic_closed_aspherical_manifolds]}  
  Theorem~\ref{the:status_of-Farrell-Jones} implies that the exotic
  aspherical manifolds mentioned in
  Subsection~\ref{subsec:exotic_asphercial_manifolds} satisfy the
  Borel Conjecture in dimension $\ge 5$ since their universal
  coverings are $\CAT(0)$-spaces. 
\end{remark}

\begin{remark}[Directed colimits of hyperbolic groups]
  \label{rem:directed_colimits_of_hyperbolic_groups}
  There are also a variety of interesting
  groups such as \emph{lacunary groups} in the sense of
  Olshanskii-Osin-Sapir~\cite{Olshanskii-Osin-Sapir(2007)} or
  \emph{groups with expanders} as they appear in the
  counterexample to the \emph{Baum-Connes Conjecture with coefficients} due
  to Higson-Lafforgue-Skandalis~\cite{Higson-Lafforgue-Skandalis(2002)}
  and which have been constructed by 
  Arzhantseva-Delzant~\cite[Theorem~7.11 and Theorem~7.12]{Arzhantseva-Delzant(2008)}.
  Since these arise as colimits of directed systems of hyperbolic groups, they
  do satisfy the Farrell-Jones Conjecture and the Borel Conjecture in dimension
  $\ge 5$ by Theorem~\ref{the:status_of-Farrell-Jones}.

  The \emph{Bost Conjecture} has also been proved for colimits of hyperbolic groups
  by Bartels-Echterhoff-L\"uck~\cite{Bartels-Echterhoff-Lueck(2008colim)}.
\end{remark}

The original source for the (Fibered) Farrell-Jones Conjecture is the
paper by Farrell-Jones~\cite[1.6 on page~257 and~1.7 on
page~262]{Farrell-Jones(1993a)}.  The $C^*$-analogue of the
Farrell-Jones Conjecture is the Baum-Connes Conjecture whose
formulation can be found in \cite[Conjecture 3.15 on page
254]{Baum-Connes-Higson(1994)}.  For more information about the
Baum-Connes Conjecture and the Farrell-Jones Conjecture and literature
about them we refer for instance to the survey
article~\cite{Lueck-Reich(2005)}.


\typeout{--------------------   Section 5 ---------------------------------------}

\section{Poincar\' e duality groups}
\label{sec:Poincare_duality_groups}

The following definition is due to Johnson-Wall~\cite{Johnson+Wall(1972)}.
\begin{definition}[Poincar\'e duality group]
\label{def:Poincare-duality_group}

A group $G$ is called a \emph{Poincar\'e duality group of dimension $n$}  
if the following conditions holds:
\begin{enumerate}
\item The group $G$ is of type FP, i.e., the trivial $\IZ G$-module $\IZ$ 
possesses a finite-dimensional
projective $\IZ G$-resolution by finitely generated projective $\IZ G$-modules;

\item We get an isomorphism of abelian groups
$$H^i(G;\IZ G) \cong 
 \left\{
 \begin{array}{ll}
 \{0\} & \text{for}\;  i \not= n;
 \\
 \IZ & \text{for}\;  i = n.
\end{array} \right.
$$
\end{enumerate}
\end{definition}

The next definition is due to Wall~\cite{Wall(1967)}.
Recall that a $CW$-complex $X$ is called \emph{finitely dominated}
if there exists a finite $CW$-complex $Y$ and 
maps $i \colon X \to Y$ and $r \colon Y \to X$ with $r \circ i \simeq \id_X$.

\begin{definition}[Poincar\'e complex]
\label{definion:Poincare_complex}
Let $X$ be a finitely dominated connected $CW$-complex with
fundamental group $\pi$. 

It is called a \emph{Poincar\'e complex of dimension $n$} 
if there exists an orientation  homomorphism $w \colon \pi \to \{\pm 1\}$ 
and an element  
$$[X] \in H_n^{\pi}(\widetilde{X};^w\IZ) = 
H_n\bigl(C_*(\widetilde{X}) \otimes_{\IZ \pi} {^w\IZ}\bigr)$$
in the $n$-th
$\pi$-equivariant homology of its universal covering $\widetilde{X}$
with coefficients in the $\IZ G$-module $^w\IZ$, 
such that the up to $\IZ \pi$-chain homotopy equivalence
unique $\IZ\pi$-chain map
$$- \cap [X] \colon C^{n-*}(\widetilde{X}) = \hom_{\IZ \pi}\bigl(C_{n-*}(\widetilde{X}),\IZ \pi\bigr) 
 \to C_*(\widetilde{X})$$
is a $\IZ \pi$-chain homotopy equivalence. Here $^w\IZ$ is the $\IZ
G$-module, whose underlying abelian group is $\IZ$ and on 
which $g \in \pi$ acts by multiplication with $w(g)$.

If in addition $X$ is a finite $CW$-complex, we call 
$X$ a \emph{finite Poincar\'e duality complex of dimension $n$}.
\end{definition}

A topological space $X$ is called an \emph{absolute neighborhood retract} or briefly
$\ANR$ if for every normal space $Z$, every closed subset $Y \subseteq Z$ and 
every (continuous) map $f \colon Y \to X$ there exists 
an open neighborhood $U$ of $Y$ in $Z$ together with an 
extension $F \colon U \to Z$ of $f$ to $U$.  A \emph{compact $n$-dimensional 
homology $\ANR$-manifold $X$} is a compact absolute neighborhood retract
such that it has a countable basis for its topology, has finite topological dimension
and for every $x \in X$ the abelian group $H_i(X,X-\{x\})$ is trivial
for $i \not= n$ and infinite cyclic for $i = n$. A closed $n$-dimensional
topological manifold is an example of a compact
$n$-dimensional homology $\ANR$-manifold 
(see~\cite[Corollary 1A in V.26 page~191]{Daverman(1986)}). 

\begin{theorem}[Homology $\ANR$-manifolds and finite Poincar\'e complexes]
\label{the:manifolds_and_Poincare}
Let $M$ be a closed topological manifold, or more generally, a compact
homology $\ANR$-manifold of dimension $n$. Then $M$ is homotopy
equivalent to a finite $n$-dimensional Poincar\'e complex.
\end{theorem}
\begin{proof}
A closed topological manifold, and more generally a compact $\ANR$, has the homotopy 
type of a finite $CW$-complex (see~\cite[Theorem~2.2]{Kirby-Siebenmann(1977)}.
\cite{West(1977)}). The usual proof of Poincar\'e duality for closed
manifolds carries over to homology manifolds.
\end{proof}

\begin{theorem}[Poincar\'e duality groups] 
\label{the:Poincare_duality_groups_versus_Poincare_complexes}
Let $G$ be a group and $n \ge 1$ be an integer. Then:

\begin{enumerate}

\item \label{the:Poincare_duality_groups_versus_Poincare_complexes:G_to_BG}
The following assertions are equivalent:
\begin{enumerate}

\item  \label{the:Poincare_duality_groups_versus_Poincare_complexes:G_to_BG:(1)}
$G$ is finitely presented and a Poincar\'e duality group of dimension $n$;

\item  \label{the:Poincare_duality_groups_versus_Poincare_complexes:G_to_BG:(2)}
There exists an $n$-dimensional aspherical Poincar\'e complex with $G$ as 
fundamental group;

\end{enumerate}

\item \label{the:Poincare_duality_groups_versus_Poincare_complexes:G_to_BG_K_0(ZG)}
Suppose that $\widetilde{K}_0(\IZ G) = 0$. Then the following assertions are equivalent:
\begin{enumerate}

\item  \label{the:Poincare_duality_groups_versus_Poincare_complexes:G_to_BG_K_0(ZG):(1)}
$G$ is finitely presented and a Poincar\'e duality group of dimension $n$;

\item  \label{the:Poincare_duality_groups_versus_Poincare_complexes:G_to_BG_K_0(ZG):(2)}
There exists a finite $n$-dimensional aspherical Poincar\'e complex with $G$ as 
fundamental group;

\end{enumerate}

\item \label{the:Poincare_duality_groups_versus_Poincare_complexes:dim_1}
A group $G$ is a Poincar\'e duality group of dimension $1$ if and only if $G \cong \IZ$;

\item \label{the:Poincare_duality_groups_versus_Poincare_complexes:dim_2}
A group $G$ is a Poincar\'e duality group of dimension $2$ if and only if $G$ is isomorphic
to the fundamental group of a closed aspherical surface;

\end{enumerate}
\end{theorem}
\begin{proof}%
\ref{the:Poincare_duality_groups_versus_Poincare_complexes:G_to_BG}
Every finitely dominated $CW$-complex has a finitely presented fundamental group
since every finite $CW$-complex has a finitely presented group and a group
which is a retract of a finitely presented group is again finitely 
presented~\cite[Lemma~1.3]{Wall(1965a)}. If there exists a
$CW$-model for $BG$ of dimension $n$, then the cohomological dimension of $G$ satisfies
$\cd(G) \le n$ and the converse is true provided that $n \ge 3$
(see~\cite[Theorem~7.1 in Chapter VIII.7 on page~205]{Brown(1982)},
\cite{Eilenberg+Ganea(1957)}, \cite{Wall(1965a)}, \cite{Wall(1966)}).
This implies that the implication
$\ref{the:Poincare_duality_groups_versus_Poincare_complexes:G_to_BG:(2)}%
\implies%
\ref{the:Poincare_duality_groups_versus_Poincare_complexes:G_to_BG:(1)}$ holds for all 
$n \ge 1$ and that the implication
$\ref{the:Poincare_duality_groups_versus_Poincare_complexes:G_to_BG:(1)}%
\implies%
\ref{the:Poincare_duality_groups_versus_Poincare_complexes:G_to_BG:(2)}$ 
holds for $n \ge 3$. For more details we 
refer to~\cite[Theorem~1]{Johnson+Wall(1972)}. 
The remaining part to show the implication
$\ref{the:Poincare_duality_groups_versus_Poincare_complexes:G_to_BG:(1)}%
\implies%
\ref{the:Poincare_duality_groups_versus_Poincare_complexes:G_to_BG:(2)}$ 
for $n = 1,2$ follows from 
assertions~\ref{the:Poincare_duality_groups_versus_Poincare_complexes:dim_1}
and~\ref{the:Poincare_duality_groups_versus_Poincare_complexes:dim_2}.
\\[1mm]%
\ref{the:Poincare_duality_groups_versus_Poincare_complexes:G_to_BG_K_0(ZG)}
This follows in dimension $n \ge 3$ from 
assertion~\ref{the:Poincare_duality_groups_versus_Poincare_complexes:G_to_BG}
and Wall's results about the finiteness obstruction which decides
whether a finitely dominated $CW$-complex is homotopy equivalent to a finite $CW$-complex
and takes values in $\widetilde{K}_0(\IZ \pi)$
(see~\cite{Ferry-Ranicki(2001),Mislin(1995),Wall(1965a),Wall(1966)}).  The implication
$\ref{the:Poincare_duality_groups_versus_Poincare_complexes:G_to_BG_K_0(ZG):(2)}~%
\implies~%
\ref{the:Poincare_duality_groups_versus_Poincare_complexes:G_to_BG_K_0(ZG):(1)}$ 
holds for all  $n \ge 1$. The remaining part to show the implication
$\ref{the:Poincare_duality_groups_versus_Poincare_complexes:G_to_BG_K_0(ZG):(1)}~%
\implies~%
\ref{the:Poincare_duality_groups_versus_Poincare_complexes:G_to_BG_K_0(ZG):(2)}$ holds
follows from assertions~\ref{the:Poincare_duality_groups_versus_Poincare_complexes:dim_1}
and~\ref{the:Poincare_duality_groups_versus_Poincare_complexes:dim_2}.
\\[1mm]%
\ref{the:Poincare_duality_groups_versus_Poincare_complexes:dim_1}
Since $S^1 = B\IZ$ is a $1$-dimensional closed manifold, $\IZ$ is a finite 
Poincare duality group of dimension $1$ by 
Theorem~\ref{the:manifolds_and_Poincare}.
We conclude from the (easy) implication
$\ref{the:Poincare_duality_groups_versus_Poincare_complexes:G_to_BG:(2)}\implies%
\ref{the:Poincare_duality_groups_versus_Poincare_complexes:G_to_BG:(1)}$ 
appearing in assertion~\ref{the:Poincare_duality_groups_versus_Poincare_complexes:G_to_BG}
that $\IZ$ is a Poincar\'e duality group of dimension $1$.
Suppose that $G$ is a Poincar\'e duality group of dimension $1$.
Since the cohomological dimension of $G$ is $1$,
it has to be a free group (see~\cite{Stallings(1968),Swan(1969)}).
Since the homology group of a group of type FP is finitely generated,
$G$ is isomorphic to a finitely generated free group
$F_r$ of rank $r$. Since $H^1(BF_r) \cong \IZ^r$ and $H_0(BF_r) \cong \IZ$, 
Poincar\'e duality can only hold for $r = 1$, i.e., $G$ is $\IZ$.
\\[1mm]%
\ref{the:Poincare_duality_groups_versus_Poincare_complexes:dim_2}
This is proved in~\cite[Theorem~2]{Eckmann-Linnell(1983)}.
See also~\cite{Bieri+Eckmann(1973),Bieri+Eckmann(1974),Eckmann(1987),Eckmann+Mueller(1980)}.
\end{proof}

\begin{conjecture}[Aspherical Poincar\'e complexes]
 \label{con:aspherical_Poincare-Complexes}
Every finite  Poincar\'e complex is homotopy equivalent to a closed manifold.
\end{conjecture}

\begin{conjecture}[Poincar'e duality groups]
 \label{con:Poincare-duality-groups}
A finitely presented 
group is a $n$-dimensional Poincar\'e duality group if and only if
it is the fundamental group of a closed $n$-dimensional topological manifold.
\end{conjecture}

Because of  Theorem~\ref{the:manifolds_and_Poincare} and 
Theorem~\ref{the:Poincare_duality_groups_versus_Poincare_complexes}~%
\ref{the:Poincare_duality_groups_versus_Poincare_complexes:G_to_BG},
Conjecture~\ref{con:aspherical_Poincare-Complexes} 
and Conjecture~\ref{con:Poincare-duality-groups} 
are equivalent.

The  \emph{disjoint disk property} says that for any $\epsilon > 0$ and maps
$f,g \colon D^2\to M$ there are maps $f',g' \colon D^2 \to M$ so that
the distance between $f$ and $f'$ and the distance between $g$ and $g'$ are bounded by $\epsilon$
and $f'(D^2) \cap g'(D^2) = \emptyset$.

\begin{lemma}\label{rem:Borel_FJC_Poincare_complexes}
Suppose that the torsionfree group $G$ and the ring $R = \IZ$
satisfy the version of the 
Farrell-Jones Conjecture stated in 
Theorem~\ref{con:Farrell-Jones_Conjecture_torsionfree}.
Let $X$ be a Poincar\'e complex of dimension $\ge 6$ with $\pi_1(X) \cong G$.
Then $X$ is homotopy equivalent to a compact homology  $\ANR$-manifold
satisfying the disjoint disk property.
\end{lemma}
\begin{proof}
See~\cite[Remark~25.13 on page~297]{Ranicki(1992)}, 
\cite[Main Theorem on page~439 and Section~8]{Bryant-Ferry-Mio-Weinberger(1996)}
and \cite[Theorem A and Theorem B]{Bryant-Ferry-Mio-Weinberger(2007)}.
\end{proof}

\begin{remark}[Compact homology $\ANR$-manifolds versus closed topological manifolds]
 \label{rem:homology_ANR-manifolds_versus_topological_manifolds}
 In the following all manifolds have dimension $\ge 6$.
 One would prefer if in the conclusion of
 Lemma~\ref{rem:Borel_FJC_Poincare_complexes} one could replace ``compact
 homology $\ANR$-manifold'' by ``closed topological manifold''.  The problem is
 that in the geometric exact surgery sequence one has to work with the
 $1$-connective cover $\bfL\langle 1 \rangle$ of the $L$-theory spectrum $\bfL$,
 whereas in the assembly map appearing in the Farrell-Jones setting one
 uses the $L$-theory spectrum $\bfL$. The $L$-theory spectrum $\bfL$
 is $4$-periodic, i.e., $\pi_n(\bfL) \cong \pi_{n+4}(\bfL)$ for $n \in
 \IZ$. The $1$-connective cover $\bfL\langle 1 \rangle$ comes with a
 map of spectra $\bff\colon \bfL\langle 1 \rangle \to \bfL$ such
 that $\pi_n(\bff)$ is an isomorphism for $n \ge 1$ and
 $\pi_n(\bfL\langle 1 \rangle) = 0$ for $n \le 0$. Since $\pi_0(\bfL)
 \cong \IZ$, one misses a part involving $L_0(\IZ)$ 
 of the so called \emph{total surgery obstruction} due to Ranicki, i.e., the
 obstruction for a finite Poincar\'e
 complex to be homotopy equivalent to a closed topological manifold, if
 one deals with the periodic $L$-theory spectrum $\bfL$ and picks up only
 the obstruction for a finite Poincar\'e complex to be homotopy
 equivalent to a compact homology $\ANR$-manifold, the so called \emph{four-periodic
 total surgery obstruction}. The difference of these two obstructions
 is related to the \emph{resolution obstruction} of Quinn which takes values in $L_0(\IZ)$. 
 Any element of $L_0(\IZ)$ can be realized by an appropriate
 compact homology $\ANR$-manifold as its \emph{resolution obstruction}. There are 
 compact homology $\ANR$-manifolds that are not homotopy equivalent to
 closed manifolds. But no example of an aspherical compact homology $\ANR$-manifold
 that is not homotopy equivalent to a closed topological manifold is known.
 For  an aspherical compact homology $\ANR$-manifold $M$, the total surgery obstruction
 and the resolution obstruction carry the same information. So we could replace
 in the conclusion of
 Lemma~\ref{rem:Borel_FJC_Poincare_complexes}  ``compact
 homology $\ANR$-manifold'' by ``closed topological manifold'' if and only if every aspherical
 compact homology $\ANR$-manifold with the disjoint disk property admits a resolution.

 We refer for instance 
 to~\cite{Bryant-Ferry-Mio-Weinberger(1996),Ferry-Pedersen(1995a),Quinn(1983a),Quinn(1987_resolution),Ranicki(1992)} for more information about this topic.
\end{remark}

\begin{question}[Vanishing of the resolution obstruction in the aspherical case]
\label{que:Vanishing_of_the_resolution_obstruction_in_the_aspherical_case}
Is every aspherical compact homology $\ANR$-manifold homotopy equivalent
to a closed manifold?
\end{question}


\typeout{--------------------   Section 6 ---------------------------------------}

\section{Product decompositions}
\label{sec:Product_decompositions}

In this section we show that, roughly speaking,
a closed aspherical manifold $M$ is a product $M_1 \times M_2$ if and
only if its fundamental group is a product $\pi_1(M) = G_1 \times
G_2$ and that such a decomposition is  unique up to homeomorphism.

\begin{theorem}[Product decomposition]
\label{the:Product_decomposition}
Let $M$ be a closed aspherical manifold of dimension $n$ with
fundamental group $G = \pi_1(M)$. Suppose we have a product decomposition
$$p_1 \times p_2 \colon G \xrightarrow{\cong}  G_1 \times G_2.$$
Suppose that $G$, $G_1$ and $G_2$  satisfy  the version of the 
Farrell-Jones Conjecture stated in
Theorem~\ref{con:Farrell-Jones_Conjecture_torsionfree}
in the case $R = \IZ$.

Then $G$, $G_1$ and $G_2$ are Poincar\'e duality groups whose
cohomological dimensions satisfy 
$$n = \cd(G) = \cd(G_1) + \cd(G_2).$$
Suppose in the sequel:

\begin{itemize}

\item the cohomological dimension $\cd(G_i)$ is different from $3$, $4$ and
      $5$ for $i = 1,2$. 
\item $n \ge 5$ or $n \le 2$ or ($n = 4$ and $G$ is good in the sense
      of Freedmann);

\end{itemize}
    
Then:

\begin{enumerate}

\item \label{the:product_decomposition:(1)}
      There are topological closed aspherical manifolds $M_1$ and
      $M_2$ together with isomorphisms 
      $$v_i \colon \pi_1(M_i) \xrightarrow{\cong}  G_i$$
      and maps
      $$f_i \colon M \to M_i$$
      for $i = 1,2$  such that
      $$f = f_1 \times f_2 \colon  M \to M_1 \times M_2$$
      is a homeomorphism and $v_i \circ  \pi_1(f_i) = p_i$ (up to inner
      automorphisms) for $i = 1,2$;

\item \label{the:product_decomposition:(2)}
      Suppose we have another such choice of topological closed aspherical manifolds $M_1'$ and
      $M_2'$ together with isomorphisms 
      $$v_i' \colon \pi_1(M_i') \xrightarrow{\cong}  G_i$$
      and maps
      $$f_i' \colon M \to M_i'$$
      for $i = 1,2$ such that the map $f' = f_1' \times f_2'$ is a homotopy
      equivalence and  $v_i' \circ \pi_1(f_i') = p_i$ (up to inner
      automorphisms) for $i = 1,2$. Then there are for $i = 1,2$
      homeomorphisms
      $h_i \colon M_i \to M_i'$
      such that $h_i \circ f_i \simeq f_i'$ and $v_i \circ \pi_1(h_i)
      = v_i'$ holds for $i = 1,2$.

\end{enumerate}
\end{theorem}
\begin{proof}
In the sequel we identify $G = G_1 \times G_2$ by $p_1 \times p_2$.
Since the closed manifold $M$ is a model for $BG$ and $\cd(G) = n$,
we can choose  $BG$ to be an $n$-dimensional  finite Poincar\'e complex in the sense of
Definition~\ref{definion:Poincare_complex} by
Theorem~\ref{the:manifolds_and_Poincare}.

From $BG = B(G_1 \times G_2) \simeq  BG_1 \times BG_2$ we conclude that
there are  finitely dominated $CW$-models for $BG_i$ for $i =1,2$.
Since $\widetilde{K}_0(\IZ G_i)$ vanishes
for $i = 0,1$ by assumption, we conclude from the theory of the
finiteness obstruction due to Wall~\cite{Wall(1965a),Wall(1966)}
that there are finite models for $BG_i$ of dimension
$\max\{\cd(G_i),3\}$. We conclude from~\cite{Gottlieb(1979)}, \cite{Quinn(1972)} 
that $BG_1$ and $BG_2$ are Poincar\'e complexes.
One easily checks using the K\"unneth formula that
$$n = \cd(G) = \cd(G_1) + \cd(G_2).$$
If $\cd(G_i) = 1$, then $BG_i$ is  homotopy
equivalent to a manifold, namely $S^1$,
by  Theorem~\ref{the:Poincare_duality_groups_versus_Poincare_complexes}~%
\ref{the:Poincare_duality_groups_versus_Poincare_complexes:dim_1}.
If $\cd(G_i) = 2$, then $BG_i$ is homotopy equivalent to a manifold
by 
Theorem~\ref{the:Poincare_duality_groups_versus_Poincare_complexes}~%
\ref{the:Poincare_duality_groups_versus_Poincare_complexes:dim_2}.
Hence it suffices to show for $i = 1,2$ that $BG_i$ is homotopy equivalent to
a closed aspherical manifold, provided that $\cd(G_i) \ge 6$.

Since by assumption $G_i$ satisfies  the version of the 
Farrell-Jones Conjecture stated in
Theorem~\ref{con:Farrell-Jones_Conjecture_torsionfree}
in the case $R = \IZ$, there exists a compact homology $\ANR$-manifold
$M_i$ that satisfies the disjoint disk property and is homotopy equivalent to $BG_i$
(see Lemma~\ref{rem:Borel_FJC_Poincare_complexes}).
Hence it remains to show
that Quinn's resolution obstruction $I(M_i) \in (1 + 8 \cdot \IZ)$ is $1$
(see~\cite[Theorem~1.1]{Quinn(1987_resolution)}). Since this
obstruction is multiplicative
(see~\cite[Theorem~1.1]{Quinn(1987_resolution)}),
we get $I(M_1 \times M_2) = I(M_1) \cdot I(M_2)$.
In general the resolution obstruction is not a homotopy
invariant, but it is known to be a homotopy invariant for 
aspherical compact $\ANR$-manifolds if the fundamental group satisfies the  Novikov
Conjecture~\ref{con:Novikov_Conjecture} (see~\cite[Proposition on
page~437]{Bryant-Ferry-Mio-Weinberger(1996)}).
Since $G_i$  satisfies  the version of the 
Farrell-Jones Conjecture stated in
Theorem~\ref{con:Farrell-Jones_Conjecture_torsionfree}
in the case $R = \IZ$, it satisfies the Novikov Conjecture
by Lemma~\ref{lem:FJC_implies_BC} and Remark~\ref{rem:The_Novikov_Conjecture_and_aspherical_manifolds}.
Hence $I(M_1 \times M_2) = I(M)$. Since $I(M)$ is a closed manifold,
we have $I(M) = 1$. Hence $I(M_i) = 1$ and $M_i$ is homotopy equivalent
to a closed manifold. This finishes the proof of
assertion~\ref{the:product_decomposition:(1)}.

Assertion~\ref{the:product_decomposition:(2)} follows from  Lemma~\ref{lem:FJC_implies_BC}.
\end{proof}

\begin{remark}[Product decompositions and non-positive sectional curvature]
\label{rem:product_decompositions_and_non-posiitve_sectional_curvature}
The following result has been proved by
Gromoll-Wolf~\cite[Theorem~2]{Gromoll-Wolf(1971)}.
Let $M$ be a closed Riemannian manifold with non-positive sectional curvature.
Suppose that we are given a splitting of its fundamental 
group $\pi_1(M) = G_1 \times G_2$ and that the center of $\pi_1(M)$ is trivial.
Then this splitting comes from an isometric product decomposition 
of closed Riemannian manifolds of non-positive sectional curvature
$M = M_1 \times M_2$. 
\end{remark}


\typeout{--------------------   Section 7 ---------------------------------------}

\section{Novikov Conjecture}
\label{sec:Novikov_Conjecture}

Let $G$ be a group and let $u \colon M \to BG$ be a map from a closed oriented
smooth manifold $M$ to $BG$. Let
$$\call(M)%
\index{L-class@$L$-class} \in \bigoplus_{k \in \IZ, k \ge 0}
H^{4k}(M;\IQ)$$ be the \emph{$L$-class of $M$}.  Its $k$-th entry
$\call(M)_k \in H^{4k}(M;\IQ)$ is a certain homogeneous polynomial of
degree $k$ in the rational Pontrjagin classes $p_i(M;\IQ) \in
H^{4i}(M;\IQ)$ for $i = 1,2, \ldots, k$ such that the coefficient
$s_k$ of the monomial $p_k(M;\IQ)$ is different from zero. The
$L$-class $\call(M)$ is determined by all the rational Pontrjagin
classes and vice versa.  The $L$-class depends on the tangent bundle and thus
on the differentiable structure of $M$.  For $x \in \prod_{k \ge 0}
H^k(BG;\IQ)$ define the \emph{higher signature of $M$ associated to
  $x$ and $u$} to be the integer
\begin{eqnarray}
  \sign_x(M,u) & := & \langle \call(M) \cup f^* x,[M] \rangle.
\end{eqnarray}
We say that $\sign_x$ for $x \in H^*(BG;\IQ)$ is \emph{homotopy invariant}
if for two closed oriented smooth manifolds $M$ and $N$ with reference
maps $u\colon M \to BG$ and $v \colon N \to BG$ we have
$$
\sign_x(M,u) = \sign_x(N,v),
$$
whenever there is an orientation preserving homotopy equivalence $f
\colon M \to N$ such that $v \circ f$ and $u$ are homotopic. If $x = 1
\in H^0(BG)$, then the higher signature $\sign_x(M,u)$ is by the
Hirzebruch signature formula
(see~\cite{Hirzebruch(1966),Hirzebruch(1970)}) 
the signature of $M$ itself and hence an invariant
of the oriented homotopy type.  This is one motivation for the
following conjecture.

\begin{conjecture}[Novikov Conjecture]
  \label{con:Novikov_Conjecture}
  Let $G$ be a group.  Then $\sign_x$ is homotopy invariant for all $x
  \in \prod_{k \in \IZ, k \ge 0} H^k(BG;\IQ)$.
\end{conjecture}

This conjecture appears for the first time in the paper by Novikov
\cite[\S 11]{Novikov(1970b)}.  A survey about its history can be found
in \cite{Ferry-Ranicki-Rosenberg(1995)}. More information can be found
for instance
in~\cite{Ferry-Ranicki-Rosenberg(1995),Ferry-Ranicki-Rosenberg(1995b),Kreck-Lueck(2005)}.

We mention the following deep result due to
Novikov~\cite{Novikov(1965c),Novikov(1965b),Novikov(1966)}.

\begin{theorem}[Topological invariance of rational Pontrjagin classes]
  \label{the:Topological_invariance_of_rational_Pontrjagin_classes}
  The rational Pontrjagin classes $p_k(M,\IQ) \in H^{4k}(M;\IQ)$ are
  topological invariants, i.e.\ for a homeomorphism $f \colon M \to N$
  of closed smooth manifolds we have
  $$H_{4k}(f;\IQ)\bigl(p_k(M;\IQ)\bigr)  =  p_k(N;\IQ)$$
  for all $k \ge 0$ and in particular $H_*(f;\IQ)(\call(M)) = \call(N)$.
\end{theorem}

The rational Pontrjagin classes are not homotopy invariants and the
integral Pontrjagin classes $p_k(M)$ are not homeomorphism invariants
(see for instance~\cite[Example~1.6 and
Theorem~4.8]{Kreck-Lueck(2005)}).

\begin{remark}[The Novikov Conjecture and aspherical manifolds]
  \label{rem:The_Novikov_Conjecture_and_aspherical_manifolds}
  Let $f \colon M \to N$ be a homotopy equivalence of closed
  aspherical manifolds.  Suppose that the Borel
  Conjecture~\ref{con:Borel_Conjecture} is true for $G = \pi_1(N)$.
  This implies that $f$ is homotopic to a homeomorphism and hence by
  Theorem~\ref{the:Topological_invariance_of_rational_Pontrjagin_classes}
  $$f_*(\call(M)) = \call(N).$$
  But this is equivalent to the conclusion of the Novikov Conjecture in
  the case $N = BG$.
\end{remark}

\begin{conjecture} \label{con:positive_scalar_curvature} A closed
  aspherical smooth manifold does not admit a Riemannian metric of
  positive scalar curvature.
\end{conjecture}

\begin{proposition} \label{prop:Strong_Novikov-Conjecture_and_pos_scalar}
  Suppose that the strong Novikov Conjecture is true for the
  group~$G$, i.e., the assembly map
  $$K_n(BG) \to K_n(C^*_r(G))$$ 
  is rationally injective for all $n \in \IZ$.  Let $M$ be a closed
  aspherical smooth manifold whose fundamental group is isomorphic to $G$.

Then $M$ carries no Riemannian metric of positive scalar curvature.
\end{proposition}
\begin{proof}
  See~\cite[Theorem~3.5]{Rosenberg(1983)}.
\end{proof}

\begin{proposition} \label{prop:Strong_Novikov-Conjecture_and_Zero_in_the_spectrum}
 Let $G$ be a group. Suppose that the assembly map
$$K_n(BG) \to K_n(C^*_r(G))$$ 
is rationally injective for all $n \in \IZ$.  Let $M$ be a closed
aspherical smooth manifold whose fundamental group is isomorphic to
$G$.

Then $M$ satisfies the Zero-in-the-Spectrum
Conjecture~\ref{con:zero-in-the-Spectrum_Conjecture}
\end{proposition}
\begin{proof}
  See~\cite[Corollary~4]{Lott(1996b)}.
\end{proof}

We refer to~\cite[Section~5.1.3]{Lueck-Reich(2005)} for a discussion
about the large class of groups for which the assembly map $K_n(BG)
\to K_n(C^*_r(G))$ is known to be injective or rationally injective.


\typeout{--------------------   Section 8 ---------------------------------------}

\section{Boundaries of hyperbolic groups}
\label{sec:Boundaries_of_hyperbolic_groups}

We announce the following two theorems 
joint with Arthur Bartels and Shmuel Weinberger.
For the notion of the boundary of a hyperbolic group
and its main properties we refer for instance 
to~\cite{Kapovich+Benakli(2002)}.

\begin{theorem}
\label{the:boundaries}
Let $G$ be a torsion-free hyperbolic group and let $n$ be an 
integer  $\geq 6$. Then: 

\begin{enumerate}

\item \label{the:boundaries:equivalent_statements}
The following statements are equivalent:
\begin{enumerate}
  \item \label{the:boundaries:equivalent_statements:sphere}
        The boundary $\partial G$ is homeomorphic to $S^{n-1}$;
  \item \label{the:boundaries:equivalent_statements:manifold}
        There is a closed aspherical topological manifold $M$ 
        such that $G \cong \pi_1(M)$, 
        its universal covering $\widetilde{M}$ is 
        homeomorphic to $\IR^n$
        and the compactification of $\widetilde{M}$ by 
        $\partial G$ is homeomorphic to $D^n$;
\end{enumerate}

\item \label{the:boundaries:uniqueness}
The aspherical manifold $M$ appearing in the assertion above
is unique up to homeomorphism.
\end{enumerate}
\end{theorem}

The proof depends strongly on the surgery theory for compact
homology $\ANR$-manifolds due to Bryant-Ferry-Mio-Weinberger~%
\cite{Bryant-Ferry-Mio-Weinberger(1996)}
and the validity of the $K$- and $L$-theoretic Farrell-Jones Conjecture for
hyperbolic groups due to
Bartels-Reich-L\"uck~\cite{Bartels-Lueck-Reich(2008hyper)}
and Bartels-L\"uck~\cite{Bartels-Lueck(2009Borelhyp)}. 
It seems likely that this result holds also if $n = 5$.
Our methods can be extended to this case 
if the surgery theory from~\cite{Bryant-Ferry-Mio-Weinberger(1996)} 
can be extended to the case of $5$-dimensional compact 
homology $\ANR$-manifolds.

We do not get information in dimensions $n \le 4$
for the usual problems about surgery. 
For instance, our methods
give no information in the case, where the boundary is homeomorphic to 
$S^3$, since virtually cyclic groups are the only hyperbolic groups
which are known to be good in the sense of 
Friedman~\cite{Freedman(1983)}. 
In the case $n = 3$ 
there is the conjecture of Cannon~\cite{Cannon(1991)} that
a group $G$ acts properly, isometrically and cocompactly on the
$3$-dimensional hyperbolic plane $\IH^3$ if and only if it is a
hyperbolic group whose boundary is homeomorphic to $S^2$. 
Provided that the infinite hyperbolic group $G$ occurs as the 
fundamental group of a closed irreducible $3$-manifold,
Bestvina-Mess~\cite[Theorem~4.1]{Bestvina-Mess(1991)} 
have shown that its universal covering is homeomorphic to 
$\IR^3$ and its compactification by $\partial G$ is 
homeomorphic to $D^3$, and the Geometrization Conjecture of
Thurston implies that $M$ is hyperbolic and 
$G$ satisfies Cannon's conjecture.
The problem is solved in the case $n = 2$, namely, 
for a hyperbolic group $G$
its boundary $\partial G$ is homeomorphic to $S^1$ if and only if
$G$ is a Fuchsian group 
(see~\cite{Casson-Jungreis(1994),Freden(1995),Gabai(1991)}).

For every $n \ge 5$ there exists a strictly negatively
curved polyhedron of dimension $n$
whose fundamental group $G$ is hyperbolic, which is homeomorphic to a
closed aspherical smooth manifold and 
whose universal covering is homeomorphic to $\IR^n$, but 
the boundary $\partial G$ is not homeomorphic to $S^{n-1}$, 
see~\cite[Theorem~5c.1 on page~384 and Remark on page~386]
{Davis-Januszkiewicz(1991)}.
Thus the condition that $\partial G$ is a sphere for a torsion-free 
hyperbolic group is (in high dimensions) not equivalent to
the existence of an aspherical manifold whose 
fundamental group is $G$.

\begin{theorem}
  \label{thm:boundary-has-spherical-cech-cohomology}
  Let $G$ be a torsion-free hyperbolic group and let $n$ be an 
integer  $\geq 6$.  Then
  \begin{enumerate}
  \item \label{thm:boundary-has-spherical-cech-cohomology:ex}
  The following statements are equivalent:
  \begin{enumerate}
    \item \label{thm:boundary-cech:cech} 
          The boundary $\partial G$ has the integral \v{C}ech cohomology of
          $S^{n-1}$;
    \item \label{thm:boundary-cech:dualtity}
          $G$ is a Poincar\'e duality group of dimension $n$;
    \item \label{thm:boundary-cech:homology-mfd}
          There exists a compact homology $\ANR$-manifold
          $M$ homotopy equivalent to $BG$. 
          In particular, $M$ is aspherical and $\pi_1(M) \cong G$;
  \end{enumerate}
  
  \item \label{thm:boundary-has-spherical-cech-cohomology:uniqueness}
  If the statements in 
  assertion~\ref{thm:boundary-has-spherical-cech-cohomology:ex}
  hold, then the compact homology $\ANR$-manifold $M$ appearing there is
  unique up to $s$-cobordism of compact $\ANR$-homology manifolds.
\end{enumerate}
\end{theorem}

The discussion of compact homology $\ANR$-manifolds versus closed topological manifolds
of Remark~\ref{rem:homology_ANR-manifolds_versus_topological_manifolds} and
Question~\ref{que:Vanishing_of_the_resolution_obstruction_in_the_aspherical_case} 
are relevant for Theorem~\ref{thm:boundary-has-spherical-cech-cohomology} as well.

In general the boundary of a hyperbolic group is not locally a 
Euclidean space but has a fractal behavior. 
If the boundary $\partial G$ of an infinite
hyperbolic group $G$ contains an open subset homeomorphic to 
Euclidean $n$-space, then it is homeomorphic to $S^n$.
This is proved in~\cite[Theorem~4.4]{Kapovich+Benakli(2002)},
where more information about the boundaries of hyperbolic groups can
be found.


\typeout{--------------------   Section 9 ---------------------------------------}

\section{$L^2$-invariants}
\label{sec:L2-invariants}

Next we mention some prominent conjectures about aspherical manifolds and
$L^2$-invariants. For more information about these conjectures and their status
we refer to~\cite{Lueck(2002)} and~\cite{Lueck(2003h)}.


\subsection{The Hopf and the Singer Conjecture}
\label{subsec:Singer_Conjecture}

\begin{conjecture}[Hopf Conjecture]%
\label{con:Hopf_Conjecture}
If  $M$ is an aspherical closed manifold of even dimension, then
$$(-1)^{\dim(M)/2} \cdot \chi(M) \ge 0.$$
If  $M$ is a  closed Riemannian manifold of even dimension with
sectional curvature $\sec(M)$, then
$$\begin{array}{rlllllll}
(-1)^{\dim(M)/2} \cdot \chi(M) & > & 0 & &
\mbox{ if } & \sec(M) & < & 0;
\\
(-1)^{\dim(M)/2} \cdot \chi(M) & \ge   & 0 & &
\mbox{ if } & \sec(M) & \le & 0;
\\
\chi(M) & = & 0 & &
\mbox{ if } & \sec(M) & = & 0;
\\
\chi(M) & \ge & 0 & &
\mbox{ if } & \sec(M) & \ge & 0;
\\
\chi(M) & > & 0 & &
\mbox{ if } & \sec(M) & > & 0.
\end{array}$$
\end{conjecture}

\begin{conjecture}[Singer Conjecture]%
\label{con:Singer_Conjecture}
If $M$ is an aspherical closed manifold, then
$$b_p^{(2)}(\widetilde{M})  =  0 \hspace{10mm}
\mbox{if } 2p \not= \dim(M).$$
If $M$ is a closed connected Riemannian manifold with negative sectional
curvature, then
$$b_p^{(2)}(\widetilde{M})   \left\{
\begin{array}{lll}
= 0 & & \mbox{if } 2p \not= \dim(M);\\
> 0 & & \mbox{if } 2p = \dim(M).
\end{array}\right.$$
\end{conjecture}


\subsection{$L^2$-torsion and aspherical manifolds}
\label{subsec:L2-torsion_and_aspherical_manifolds}

\begin{conjecture}[$L^2$-torsion for aspherical manifolds]%
\label{con:L2-torsion_for_aspherical_manifolds}
If $M$ is an aspherical closed manifold of odd dimension, then
$\widetilde{M}$ is $\det$-$L^2$-acyclic  and
$$(-1)^{\frac{\dim(M)-1}{2}} \cdot \rho^{(2)}(\widetilde{M})  \ge 
0.$$
If $M$ is a closed connected Riemannian manifold of odd dimension with negative
sectional curvature, then
$\widetilde{M}$ is $\det$-$L^2$-acyclic  and
$$(-1)^{\frac{\dim(M)-1}{2}} \cdot \rho^{(2)}(\widetilde{M})  > 
0.$$
If $M$ is an aspherical closed manifold
whose fundamental group contains an amenable infinite normal subgroup,
then $\widetilde{M}$ is $\det$-$L^2$-acyclic and
$$\rho^{(2)}(\widetilde{M})  =  0.$$
\end{conjecture}


\subsection{Simplicial volume and $L^2$-invariants}
\label{subsec:simplicial_volume_and_L2-invariants}

\begin{conjecture}[Simplicial volume and $L^2$-invariants]
\label{con:simplicial_volume_and_L2-invariants}
Let $M$ be an aspherical closed orientable manifold.
Suppose that its simplicial volume $||M||$ vanishes. Then
$\widetilde{M}$ is of determinant class
and
\begin{eqnarray*}
b_p^{(2)}(\widetilde{M}) & = & 0 \hspace{5mm} \mbox{ for } p \ge 0;
\\
\rho^{(2)}(\widetilde{M}) & = & 0.
\end{eqnarray*}
\end{conjecture}


\subsection{Zero-in-the-Spectrum Conjecture}
\label{subsec:zero-in-the-Spectrum_Conjecture}

\begin{conjecture}[Zero-in-the-spectrum Conjecture]%
\label{con:zero-in-the-Spectrum_Conjecture}
Let $\widetilde{M}$ be a complete Riemannian manifold.
Suppose that $\widetilde{M}$ is the universal covering of an
aspherical closed Riemannian manifold $M$ (with the Riemannian metric
coming from $M$). Then for some $p \ge 0$ zero is  in the
Spectrum of the minimal closure
$$(\Delta_p)_{\min}\colon  \dom\bigl((\Delta_p)_{\min}\bigr) \subset
L^2\Omega^p(\widetilde{M}) \to L^2\Omega^p(\widetilde{M})$$ of the
Laplacian acting on smooth $p$-forms on $\widetilde{M}$.
\end{conjecture}

\begin{remark}[Non-aspherical counterexamples to the Zero-in-the-Spectrum Conjecture]
\label{rem:non_aspherical_counterexamples_to_the_ZISC}
For all of the conjectures about aspherical spaces stated in this article it is obvious 
that they cannot be true if one drops the condition aspherical except for the  
zero-in-the-Spectrum Conjecture~\ref{con:zero-in-the-Spectrum_Conjecture}. 
Farber and Weinberger~\cite{Farber-Weinberger(2001)} 
gave the first example of a
closed Riemannian manifold for which zero is not in the spectrum
of the minimal closure
$(\Delta_p)_{\min}\colon  \dom\left((\Delta_p)_{\min}\right) \subset
L^2\Omega^p(\widetilde{M}) \to L^2\Omega^p(\widetilde{M})$ of the
Laplacian acting on smooth $p$-forms on $\widetilde{M}$ for each $p \ge 0$.
The construction by Higson, Roe and Schick~\cite{Higson-Roe-Schick(2001)}
yields a plenty of such counterexamples. But there are no aspherical
counterexamples known.
\end{remark}


\typeout{--------------------   Section 10 ---------------------------------------}

\section{The universe of closed  manifolds}
\label{sec:The_universe_of_closed_manifolds}

At the end we describe (winking) our universe of closed manifolds.

The idea of a random group has successfully been used to construct
groups with certain properties, see for
instance~\cite{Arzhantseva-Delzant(2008)}, \cite{Ghys(2004)},
\cite[9.B on pages273ff]{Gromov(1993)},
\cite{Gromov(2003)}, \cite{Ollivier(2005)},\cite{Pansu(2003)},
\cite{Silberman(2003)}  and~\cite{Zuk(2003)}.   
In a precise statistical sense almost all finitely
presented groups are hyperbolic see~\cite{Olshanskii(1992)}.
One can actually show that in a
precise statistical sense almost all finitely presented groups are
torsionfree hyperbolic and in particular have a finite model for their
classifying space. In most cases it is given by the limit for $n \to \infty$ of
the quotient of the number of finitely presented groups with a certain
property (P) which are given by a presentation satisfying a certain
condition $C_n$ by the number of all finitely presented groups which
are given by a presentation satisfying  condition $C_n$.

It is not clear what it means in a precise sense to talk about a
random closed manifold.  Nevertheless, the author's intuition is that almost all
closed manifolds are aspherical.  (A related question would be whether
a random closed smooth manifold admits a Riemannian metric with
non-positive sectional curvature.)  This intuition is supported by
Remark~\ref{rem:Characteristic_numbers_and_aspherical_manifolds}. It
is certainly true in dimension $2$ since only finitely many closed
surfaces are not aspherical.  The characterization of closed
$3$-dimensional manifolds in Subsection~\ref{subsec:Low-dimensions}
seems to fit as well. In the sequel we assume that this (vague) intuition is
correct.

If we combine these considerations, we get that almost all closed
manifolds are aspherical and have a hyperbolic fundamental group.
Since except in dimension $4$ the Borel Conjecture is known in this
case by Lemma~\ref{lem:FJC_implies_BC},
Remark~\ref{rem:The_Borel_Conjecture_in_low_dimensions}
and  Theorem~\ref{the:status_of-Farrell-Jones}, we get as a
consequence that almost almost all closed manifolds are aspherical and
topologically rigid.

A closed manifold $M$ is called \emph{asymmetric} if every finite
group which acts effectively on $M$ is trivial.  This is equivalent to
the statement that for any choice of Riemannian metric on $M$ the
group of isometries is trivial (see~\cite[Introduction]{Kreck(2008)}).
A survey on asymmetric closed
manifolds can be found in~\cite{Puppe(2007)}. 
The first constructions of asymmetric closed aspherical
manifolds are due to Connor-Raymond-Weinberger~\cite {Conner+Raymond+Weinberger(1972)}.
The first simply-connected asymmetric manifold has been
constructed by Kreck~\cite{Kreck(2008)} answering a question of
Raymond and Schultz~\cite[page~260]{Browder-Hsiang(1978)} which was repeated
by Adem and Davis~\cite {Adem-Davis(2002)} in their problem list. 
Raymond and Schultz expressed also their feeling that a random manifold
should be asymmetric. Borel has shown that
an aspherical closed manifold is asymmetric 
if its fundamental group is
centerless and its outer automorphism group is torsionfree
(see the manuscript ``On periodic maps of certain $K(\pi,1)$'' 
in~\cite[pages~57--60]{Borel(1983_colIII)}).

This leads to the intuitive statement:
\begin{quote}
  Almost all closed manifolds are aspherical, topologically
  rigid and asymmetric.
\end{quote}

In particular almost every closed manifold is determined up to
homeomorphism by its fundamental group.

This is \---- at least on the first glance \---- surprising since
often our favorite manifolds are not asymmetric and not determined by
their fundamental group. There are prominent manifolds such as lens
spaces which are homotopy equivalent but not homeomorphic.  There seem
to be plenty of simply connected manifolds. So why do human beings
may have the feeling that the universe of closed manifolds described above
is different from their expectation?

If one asks people for the most prominent closed manifold, most people
name the standard sphere. It is interesting that the $n$-dimensional
standard sphere $S^n$ can be characterized among (simply connected) closed Riemannian
manifolds of dimension $n$ by the property that its isometry group has
maximal dimension. More precisely, if $M$ is a closed $n$-dimensional smooth manifold,
then the dimension of its isometry group for any Riemannian metric is bounded
by $n(n+1)/2$ and the maximum $n(n+1)/2$ is attained if and only if
$M$ is diffeomorphic to $S^n$ or $\IR\IP^n$; see Hsiang~\cite{Hsiang(1967)}, 
where the Ph.D-thesis of Eisenhart is cited and the dimension of the isometry group of
exotic spheres is investigated. It is likely that the
human taste whether a geometric object is beautiful is closely related
to the question how many symmetries it admits. In general it seems to
be the case that a human being is attracted by unusual representatives
among mathematical objects such as groups or closed manifolds and not
by the generic ones.  In group theory it is clear that random groups
can have very strange properties and that these groups are to some
extend scary. The analogous statement seems to hold for closed
topological manifolds.

At the time of writing the author cannot really name a group which
could be a potential counterexample to the Farrell-Jones Conjecture or
other conjectures discussed in this article. But the author has the
feeling that nevertheless the class of groups, for which we can prove
the conjecture and which is for ``human standards'' quite large, is
only a very tiny portion of the whole universe of groups and the
question whether these conjectures are true for all groups is
completely open.

Here is an interesting parallel to our actual universe.  If you
materialize at a random point in the universe it will be very cold and
nothing will be there. There is no interaction between different
random points, i.e., it is rigid.  A human being will not like this
place, actually even worse, it cannot exist at such a random place.
But there are unusual rare non-generic points in the universe, where
human beings can exist such as the surface of our planet and there a
lot of things and interactions are happening.  And human beings tend
to think that the rest of the universe looks like the place they are
living in and cannot really comprehend the rest of the universe.



\typeout{-------------------- References -------------------------------}



\end{document}